\documentclass[11 pt]{amsart}
\usepackage{full page}
\usepackage{graphicx}
\usepackage{amsmath}
\usepackage{amssymb}
\usepackage{mathrsfs}
\usepackage{bbm}
\usepackage{tikz-cd}
\usepackage{enumerate}
\usepackage{url}
\usepackage{amsthm}
\usepackage{mathdots}
\definecolor{mypink}{cmyk}{0, 0.7808, 0.4429, 0.1412}
\definecolor{lightblue}{cmyk}{1,.4,.4,.1}
\definecolor{myred}{cmyk}{0,.82,.87,.25}
\definecolor{myblue}{cmyk}{.81,.41,0,.09}
\usepackage[colorlinks = true,
linkcolor = black,
urlcolor = mypink,
citecolor = myblue]{hyperref}
\usetikzlibrary{cd}
\title{$\mathbb{A}^1$-homotopy equivalences and a theorem of Whitehead}
\author{Eoin Mackall}
\email{eoinmackall \emph{at} gmail.com}
\urladdr{\url{www.eoinmackall.com}}
\date{\today}
\newtheorem{thm}{Theorem}[section]

\newtheorem{prop}[thm]{Proposition}
\newtheorem{cor}[thm]{Corollary}
\newtheorem{lem}[thm]{Lemma}

\theoremstyle{definition}
\newtheorem{defn}[thm]{Definition}
\newtheorem{exmp}[thm]{Example}
\newtheorem{rmk}[thm]{Remark}

\numberwithin{thm}{subsection}

\newcounter{item}
\newcommand{\ite}[1]{\refstepcounter{item}\label{#1}}

\newcommand{\CH}{\mathrm{CH}}
\begin{document}
\begin{abstract}
We prove analogs of Whitehead's theorem (from algebraic topology) for both the Chow groups and for the Grothendieck group of coherent sheaves: a morphism between smooth projective varieties whose pushforward is an isomorphism on the Chow groups, or on the Grothendieck group of coherent sheaves, is an isomorphism. As a corollary, we show that there are no nontrivial naive $\mathbb{A}^1$-homotopy equivalences between smooth projective varieties.
\end{abstract}\maketitle

\section{Introduction}
$\mathbb{A}^1$-homotopy theory is a contemporary subject that applies homotopic techniques to algebraic varieties by working with a model category structure on a (twice localized) collection of simplicial presheaves of varieties \cite{MR1813224, MR3534540}. Recent works (e.g.\ \cite{MR2803793, MR3252968, MR4048668}) have used this approach both to classify $\mathbb{A}^1$-weak equivalence classes of varieties and to give new insights on classical problems from algebraic geometry. By contrast, there is a more hands-on variant of $\mathbb{A}^1$-homotopy theory, introduced in either \cite{10.2307/j.ctt7tcnh} or \cite{52}, that is both considerably more naive than the model category construction of $\mathbb{A}^1$-homotopy theory and considerably less studied. 

In this naive variant, an $\mathbb{A}^1$-homotopy is defined on the group of finite correspondences between smooth varieties (see Definition \ref{htdef} below); already in this framework the product $X\times \mathbb{A}^1$ is naively $\mathbb{A}^1$-homotopy equivalent with $X$ for any smooth variety $X$. This text stemmed from studying what properties of a variety are preserved under this naive version of $\mathbb{A}^1$-homotopy equivalence. Our main result, in this regard, is Corollary \ref{last} that shows: if $X$ and $Y$ are smooth projective varieties that are naively $\mathbb{A}^1$-homotopy equivalent, then $X$ and $Y$ are isomorphic.

There are two observations that go into the proof of Corollary \ref{last}. The first observation, that comprises most of Section \ref{s: cor} and Section \ref{s: hc}, is that the naive definition of $\mathbb{A}^1$-homotopy descends to the level of rational equivalence classes of cycles. To be precise, we show that if two finite correspondences are $\mathbb{A}^1$-homotopic then they have the same rational equivalence class in the Chow ring. In particular, the morphisms that these correspondences induce on Chow groups are identical. This is akin to the situation in ordinary algebraic topology where homotopic maps induce identical pushforwards on homology.

The second observation that goes into the proof of Corollary \ref{last} is an analog of Whitehead's theorem for the Chow groups of smooth projective varieties. To explain this analogy, recall that the classical version of Whitehead's theorem (e.g.\ from \cite[Theorem 4.5]{MR1867354}) is the statement: a continuous map $f:X\rightarrow Y$ between connected CW complexes $X$ and $Y$ with the property that the pushforward on homotopy groups $$f_*:\pi_n(X)\rightarrow \pi_n(Y)$$ is an isomorphism for every $n\geq 1$ is a homotopy equivalence. Together with Hurewicz's theorem (e.g.\ from \cite[Corollary 4.33]{MR1867354}), Whitehead's theorem implies that a continuous map $f:X\rightarrow Y$ between simply connected CW complexes $X$ and $Y$ with the property that the pushforward on singular homology $$f_*:H_n(X, \mathbb{Z})\rightarrow H_n(Y,\mathbb{Z})$$ is an isomorphism for all $n\in \mathbb{Z}$ is a homotopy equivalence.

Our analog of Whitehead's theorem for Chow groups (Theorem \ref{wht} below) is the following direct generalization: a morphism $f:X\rightarrow Y$ between two smooth projective varieties $X$ and $Y$ has the property that the pushforward on Chow groups $$f_*:\mathrm{CH}_n(X)\rightarrow \mathrm{CH}_n(Y)$$ is an isomorphism for all $n\in \mathbb{Z}$ if and only if $f$ is an isomorphism. The proof of this statement uses only elementary properties of Chow groups together with an application of Zariski's Main Theorem. Following the same argument, we also improve on the classical variant for singular homology by showing (in Corollary \ref{sing}) that a morphism $f:X\rightarrow Y$, between smooth projective complex varieties, that induces an isomorphism on the singular homology of the underlying complex manifolds $$f_*:H_n(X(\mathbb{C}),\mathbb{Z})\rightarrow H_n(Y(\mathbb{C}),\mathbb{Z})$$ for all even integers $n$, is an isomorphism.

Essentially all of the results that we've mentioned so far (for cycles and Chow groups) also hold for the Grothendieck group of coherent sheaves. Throughout this text, we formulate and prove results for both the Chow groups and the Grothendieck group of coherent sheaves side-by-side to emphasize their similarity. In particular we introduce a notion of $\mathbb{A}^1$-homotopy for the $G$-correspondences considered in \cite{Manin_1968}, we show that naively $\mathbb{A}^1$-homotopic $G$-correspondences induce identical maps on $G$-theory, and we prove an analog of Whitehead's theorem for the Grothendieck group of coherent sheaves of smooth and projective varieties.

Lastly, we note that it should be possible to prove the results obtained here in a unified way through the use of Voevodsky's category of motives with compact support. However, in the current state of affairs, this would mean that we would need to localize at $p$ whenever we wanted to make statements over an imperfect field of characteristic $p$ as it's not known, at the moment, whether integral Borel--Moore motivic homology agrees with integral higher Chow groups in this setting (see \cite[Section 5]{MR3548464} for a convenient summary of known results in this regard). For this reason, it would seem that the content here is both simpler and more general than the corresponding statements made at the level of geometric motives with compact support.

We turn now to an outline of this text. Section \ref{s: cor} is mainly preliminary. There are a number of places where it's important to remember when a construction happens at the level of cycles or, at the level of cycle classes and we recall them all, in detail, in this section. 

Section \ref{s: hc} is the most technical section of this text because, in order to show that $\mathbb{A}^1$-homotopic finite correspondences have the same cycle class, we need to be able to compose at the level of cycle classes. This means that we have to compactify the finite correspondence that realizes an explicit $\mathbb{A}^1$-homotopy and check that we haven't altered the restrictions along the closed immersion at 0 and 1 under this compactification.

Section \ref{s: wht} is devoted to analogs of Whitehead's theorem for both the Chow groups and for the Grothendieck group of coherent sheaves. We proceed here by proving general results about the structure of a morphism between projective varieties depending on properties of the resulting pushforward on either the Chow groups of the Grothendieck group. Combining all of these results with Zariski's Main Theorem gives our two analogs of Whitehead's theorem.\\

\noindent\textbf{Notation and Conventions}. In this text, a variety is a separated and geometrically integral scheme of finite type over an arbitrary (but fixed) base field $k$.

\section{Recollections on correspondences}\label{s: cor}
This section is preliminary. Here we recall the definitions and properties of finite correspondences, Chow correspondences, and $G$-correspondences that are used in Section \ref{s: hc}. We also include here the definition of $\mathbb{A}^1$-homotopy for finite correspondences and we introduce corresponding notions of $\mathbb{A}^1$-homotopy for both Chow correspondences and $G$-correspondences.

Throughout this section we write $Z_r(X)$ for the free abelian group of dimension-$r$ cycles on a variety $X$; we write $\CH_r(X)$ for the Chow group of dimension-$r$ cycles on $X$, i.e.\ the quotient of $Z_r(X)$ by the subgroup of cycles rationally equivalent to zero; we write $G(X)$ for the Grothendieck of group of coherent sheaves on $X$. When appropriate, we may omit the subscript and simply write $\CH(X)$ to mean the direct sum of the groups $\CH_i(X)$ over all integers $i\in \mathbb{Z}$. Finally, we let $X$ and $Y$ be two arbitrary smooth varieties defined over the base field $k$.

\subsection{Finite Correspondences}
Recall (from \cite[Definition 1.1]{52}) that an \textit{elementary correspondence} from $X$ to $Y$ is an integral subscheme $W\subset X\times Y$ with finite and surjective projection to $X$. A \textit{finite correspondence} from $X$ to $Y$ is a formal integral linear combination of elementary correspondences. In this text, we write $\mathrm{Cor}(X,Y)$ for the group of all finite correspondences from $X$ to $Y$. We remark there is an obvious inclusion
\ite{it: fc}\[\tag{no.\ref{it: fc}} \mathrm{Cor}(X,Y)\subset Z_{\mathrm{dim}(X)}(X\times Y)\] given by considering any finite correspondence as a cycle on $X\times Y$.

If $Z$ is another smooth variety, then there is a well-defined composition of finite correspondences\ite{it: comp}\begin{gather*} \tag{no.\ref{it: comp}} \mathrm{Cor}(X,Y)\times \mathrm{Cor}(Y,Z)\rightarrow \mathrm{Cor}(X,Z)\\ (\alpha,\beta)\mapsto \alpha\circ \beta\end{gather*} defined on elementary correspondences $\alpha=V$ and $\beta=W$ as the pushforward cycle, along the projection $X\times Y\times Z\rightarrow X\times Z$, of the intersection cycle $(V\times Z) \cdot (X\times W)$ on $X\times Y \times Z$. For the definition of the intersection cycle see \cite[Definition 17A.1]{52}; note that both the intersection cycle and the pushforward cycle are well-defined since $V\times Z$ and $X\times W$ intersect properly and each irreducible component of the intersection cycle $(V\times Z)\cdot (X\times W)$ has closed image in $X\times Y$ by \cite[Lemma 1.7]{52}. Composition for arbitrary finite correspondences is then defined distributively.

A natural source of correspondences from $X$ to $Y$ are morphisms. For any morphism $f:X\rightarrow Y$ one can associate the graph $\Gamma_f\subset X\times Y$ of $f$ as a subvariety of the product $X\times Y$. This gives an inclusion \ite{it: graph}\begin{align*}\tag{no.\ref{it: graph}} \mathrm{Hom}(X,Y)&\subset \mathrm{Cor}(X,Y)\\ f&\mapsto \Gamma_f \end{align*} that behaves well with finite correspondences in the sense that if $Z$ is another smooth variety then, for any two morphisms $f:X\rightarrow Y$ and $g:Y\rightarrow Z$, one has an equality \ite{it: gc}\[\tag{no.\ref{it: gc}} \Gamma_g\circ \Gamma_f = \Gamma_{g\circ f}\] as finite correspondences in $\mathrm{Cor}(X,Z)$. In the special case where $Y=X$, the identity morphism for $X$ is identified, under the inclusion (no.\ref{it: graph}), with the diagonal \ite{it: id}\[\tag{no.\ref{it: id}} X=\Delta_X\subset X\times X\] considered as a finite correspondence from $X$ to itself.

In this text we'll be concerned with the notions of $\mathbb{A}^1$-homotopic finite correspondences and with naive $\mathbb{A}^1$-homotopy equivalences of smooth varieties:

\begin{defn}[{\cite[Definition 2.25]{52}}]\label{htdef}
	We say that finite correspondences $\alpha,\beta \in \mathrm{Cor}(X,Y)$ are $\mathbb{A}^1$-\textit{homotopic} if there exists a finite correspondence $h\in \mathrm{Cor}(X\times \mathbb{A}^1, Y)$ and equalities of finite correspondences in $\mathrm{Cor}(X,Y)$ \[\alpha = h \circ \Gamma_{i_0}\quad \mbox{and} \quad \beta= h \circ \Gamma_{i_1}\] where the morphisms \[i_0:X= X\times \{0\}\rightarrow X\times \mathbb{A}^1 \quad \mbox{and} \quad i_1:X= X\times \{1\}\rightarrow X\times \mathbb{A}^1\] are the corresponding inclusions.
	
	We say that two smooth varieties $X$ and $Y$ are \textit{naively} $\mathbb{A}^1$-\textit{homotopy equivalent} if there exist morphisms $f:X\rightarrow Y$ and $g:Y\rightarrow X$ such that the graph $\Gamma_{g\circ f}$ is $\mathbb{A}^1$-homotopic to the diagonal $\Delta_X$ and the graph $\Gamma_{f\circ g}$ is $\mathbb{A}^1$-homotopic to the diagonal $\Delta_Y$.
\end{defn}

\subsection{Chow Correspondences}
In this text a \textit{Chow correspondence} from $X$ to $Y$ is an element of $\CH_{\mathrm{dim}(X)}(X\times Y)$. Typically, a Chow correspondence is just called a correspondence in the literature (cf.\ \cite[Definition 16.1.1]{MR1644323}, \cite[Definition 62.1]{MR2427530}) but, we add Chow here to distinguish the notion from both finite correspondences and $G$-correspondences. Because of the inclusion (no.\ref{it: fc}), we get a map \ite{it: corcor}\[\tag{no.\ref{it: corcor}}\mathrm{Cor}(X,Y)\rightarrow \mathrm{CH}_{\mathrm{dim}(X)}(X\times Y)\] that associates to any finite correspondence its rational equivalence class.

Similar to finite correspondences, Chow correspondences have a well-defined composition between proper varieties. That is to say, if $X$, $Y$, and $Z$ are smooth varieties and if $Y$ is proper, then there is a composition \ite{it: ccomp}\begin{gather*}\tag{no.\ref{it: ccomp}} \mathrm{CH}_{\mathrm{dim}(X)}(X\times Y)\times \CH_{\mathrm{dim}(Y)}(Y\times Z)\rightarrow \CH_{\mathrm{dim}(X)}(X\times Z)\\ (\alpha,\beta)\mapsto \alpha \circ \beta\end{gather*} defined for any two Chow correspondences $\alpha$ and $\beta$ as the proper pushforward along the projection $X\times Y \times Z\rightarrow X\times Z$ of the intersection product $(\alpha \times [Z]) \cdot([X]\times \beta)$ in $\CH_{\mathrm{dim}(X)}(X\times Y \times Z)$.

By construction, composition of finite correspondences and composition of Chow correspondences commute with the maps (no.\ref{it: corcor}) assigning a finite correspondence to its cycle class. That is to say, the cycle class of the composition of two finite correspondences is the composition of the associated Chow correspondences whenever both compositions are defined. This follows from \cite[Chapter 8 Section 2, Chapter 20 Section 4, and Example 7.1.2]{MR1644323}.

Chow correspondences define morphisms on the Chow groups of varieties. If $X$ is proper and if $Z$ is a smooth variety, then for any Chow correspondence $\alpha \in \CH_{\mathrm{dim}(X)}(X\times Y)$ one can associate the morphism \ite{it: morph}\begin{gather*}\tag{no.\ref{it: morph}}\alpha_*:\CH(Z\times X)\rightarrow \CH(Z\times Y)\\ \beta\mapsto \pi_{ZY*}(\pi_{YX}^*(\alpha) \cdot \pi_{ZX}^*(\beta))\end{gather*} defined on a cycle class $\beta\in \CH(Z\times X)$ as the proper pushforward, along the projection map $\pi_{ZY}:Z\times Y\times X\rightarrow Z\times Y$, of the intersection product of the flat pullback $\pi_{YX}^*(\alpha)$, along the projection $\pi_{YX}:Z\times Y\times X\rightarrow Y\times X$, and the flat pullback $\pi_{ZX}^*(\beta)$, along the projection $\pi_{ZX}: Z\times Y\times X\rightarrow Z\times X$. If $\alpha=[\Gamma_f]$ is the cycle class of the graph of a morphism $f:X\rightarrow Y$, then one has \ite{it: push}\[\tag{no.\ref{it: push}} [\Gamma_f]_\ast(\beta) = (\mathrm{Id}_Z\times f)_*(\beta)\] for every $\beta\in \CH(Z\times X)$ by \cite[Proposition 62.4]{MR2427530} or \cite[Proposition 16.1.1]{MR1644323}. In a similar fashion, if $Y$ is proper then for any Chow correspondence $\alpha\in \CH_{\mathrm{dim}(X)}(X\times Y)$ and for any smooth variety $Z$ one can associate the morphism \ite{it: morph2}\begin{gather*}\tag{no.\ref{it: morph2}} \alpha^*: \mathrm{CH}(Y\times Z)\rightarrow \mathrm{CH}(X\times Z)\\ \beta \mapsto \pi_{XZ*}(\pi_{XY}^*(\alpha) \cdot \pi_{YZ}^*(\beta))\end{gather*} with maps $\pi_{XY}$, $\pi_{XZ}$, and $\pi_{YZ}$ defined similar to the above (but, with components in reverse order). If $\alpha=[\Gamma_f]$ is the cycle class of the graph of a morphism $f:X\rightarrow Y$, then \ite{it: pull}\[\tag{no.\ref{it: pull}} [\Gamma_f]^*(\beta) =(f\times \mathrm{Id}_Z)^*(\beta)\] for every $\beta\in \CH(Y\times Z)$ by \cite[Proposition 62.4]{MR2427530} or \cite[Proposition 16.1.1]{MR1644323}. Note that one has an equality \ite{it: gys}\[\tag{no.\ref{it: gys}} \alpha^*(\beta)= \beta \circ \alpha\] for any pair of Chow correspondences $\alpha\in \mathrm{CH}_{\mathrm{dim}(X)}(X\times Y)$ and $\beta\in \mathrm{CH}_{\mathrm{dim}(Y)}(Y\times Z)$.

\begin{defn}
We say that Chow correspondences $\alpha,\beta \in \CH_{\mathrm{dim}(X)}(X\times Y)$ are \textit{$\mathbb{A}^1$-homotopic} if there exists a Chow correspondence $h\in \CH_{\mathrm{dim}(X)+1}(X\times \mathbb{A}^1\times Y)$ and equalities \[\alpha= i_0^*(h) \quad \mbox{and}\quad \beta=i_1^*(h)\] of Gysin pullbacks along the regular closed embeddings \[i_0:X= X\times \{0\}\times Y\rightarrow X\times \mathbb{A}^1\times Y \quad \mbox{and} \quad i_1:X= X\times \{1\}\times Y\rightarrow X\times \mathbb{A}^1\times Y.\]
\end{defn}

\begin{lem}\label{a1chow}
Chow correspondences $\alpha,\beta\in \CH(X\times Y)$ are $\mathbb{A}^1$-homotopic if and only if $\alpha=\beta$.
\end{lem}

\begin{proof}
Let $\pi:X\times \mathbb{A}^1\times Y\rightarrow X\times Y$ denote the projection. Then, by the functorality of pullbacks, \[i_0^*\circ \pi^*(x) = \mathrm{Id}_{X\times Y}(x)=i_1^*\circ \pi^*(x)\] for every $x\in \CH(X\times Y)$ so that the reverse direction follows by setting $h=\pi^*(\alpha)=\pi^*(\beta)$. The forward direction is equally clear, noting that $\pi^*:\CH(X\times Y)\rightarrow \CH(X\times \mathbb{A}^1 \times Y)$ is surjective.
\end{proof}

\subsection{G Correspondences} Following \cite{Manin_1968}, we define a \textit{$G$-correspondence} from $X$ to $Y$ to be an element of $G(X\times Y)$. Similar to Chow correspondences, there is a well-defined composition of $G$-correspondences for proper varieties. More precisely, if $X$, $Y$ and $Z$ are smooth varieties and if $Y$ is proper, then there is a composition \ite{it: kcomp}\begin{gather*}\tag{no.\ref{it: kcomp}} G(X\times Y)\times G(Y\times Z)\rightarrow G(X\times Z)\\ (\alpha,\beta)\mapsto \alpha \circ \beta\end{gather*} defined for any two $G$-correspondences $\alpha$ and $\beta$ as the proper pushforward, along the projection $\pi_{XZ}:X\times Y \times Z\rightarrow X\times Z$, of the product $\pi_{XY}^*(\alpha)\cdot \pi^*_{YZ}(\beta)$ inside of $G(X\times Y \times Z)$ of the pullbacks along the projections $\pi_{XY}:X\times Y \times Z\rightarrow X\times Y$ and $\pi_{YZ}:X\times Y \times Z \rightarrow Y\times Z$.

Similar to a Chow correspondence, a $G$-correspondence also defines pushforward and pullback morphisms. That is to say, if $X$ is proper and $Z$ is any smooth variety then for any $G$-correspondence $\alpha \in G(X\times Y)$ one can associate the pushforward morphism\ite{it: kmorph}\begin{gather*}\tag{no.\ref{it: kmorph}}\alpha_*:G(Z\times X)\rightarrow G(Z\times Y)\\ \beta\mapsto \pi_{ZY*}(\pi_{YX}^*(\alpha) \cdot \pi_{ZX}^*(\beta)).\end{gather*} If instead $Y$ is proper, then there are pullback morphisms\ite{it: kmorph2}\begin{gather*}\tag{no.\ref{it: kmorph2}} \alpha^*: G(Y\times Z)\rightarrow G(X\times Z)\\ \beta \mapsto \pi_{XZ*}(\pi_{XY}^*(\alpha) \cdot \pi_{YZ}^*(\beta)).\end{gather*} In the above all maps are defined as in the previous subsection. As before, there are identities when $\alpha=[\mathcal{O}_{\Gamma_f}]$ is the class of the structure sheaf of the graph of a morphism $f:X\rightarrow Y$ (see \cite[Section 3, Corollary]{Manin_1968})\ite{it: kpush}\[\tag{no.\ref{it: kpush}} [\Gamma_f]_\ast(\beta) = (\mathrm{Id}_Z\times f)_*(\beta)\] for every $\beta\in G(Z\times X)$ and \ite{it: kpull}\[\tag{no.\ref{it: kpull}} [\Gamma_f]^*(\beta) =(f\times \mathrm{Id}_Z)^*(\beta)\] for every $\beta \in G(Y\times Z)$. Finally, we note that one has an equality  \ite{it: kgys}\[\tag{no.\ref{it: kgys}} \alpha^*(\beta)= \beta \circ \alpha\] for any pair of $G$-correspondences $\alpha\in G(X\times Y)$ and $\beta\in G(Y\times Z)$.

\begin{defn}
We say that $G$-correspondences $\alpha,\beta \in G(X\times Y)$ are \textit{$\mathbb{A}^1$-homotopic} if there exists a $G$-correspondence $h\in G(X\times \mathbb{A}^1\times Y)$ and equalities \[\alpha= i_0^*(h) \quad \mbox{and}\quad \beta=i_1^*(h)\] of pullbacks along the regular closed embeddings \[i_0:X= X\times \{0\}\times Y\rightarrow X\times \mathbb{A}^1\times Y \quad \mbox{and} \quad i_1:X= X\times \{1\}\times Y\rightarrow X\times \mathbb{A}^1\times Y.\]
\end{defn}

The following proof is identical to that of Lemma \ref{a1chow}.

\begin{lem}\label{a1k}
$G$-correspondences $\alpha,\beta \in G(X\times Y)$ are $\mathbb{A}^1$-homotopic if and only if $\alpha=\beta$.$\hfill\square$
\end{lem}

\section{Homotopic correspondences}\label{s: hc}
In this section we compare the notions of $\mathbb{A}^1$-homotopy for both finite and Chow correspondences. The main result of this section, saying that $\mathbb{A}^1$-homotopic finite correspondences are $\mathbb{A}^1$-homotopic as Chow correspondences, appears as Proposition \ref{ht} below. The proof of this result (especially Lemma \ref{exists} below) is closely related to, and should follow from, the equivalence between some of Bloch's higher Chow groups \cite{MR852815} and the usual Chow groups; we've chosen to spell out the details here since they are absent from the classical sources.

In this section we also extend the definition of naive $\mathbb{A}^1$-homotopy equivalence to a notion of naive $\mathbb{A}^1$-$G$-homotopy equivalence using our naive definition of $\mathbb{A}^1$-homotopy between $G$-correspondences given in Section \ref{s: cor}. This allows us to formally mimic the results we obtain for the Chow groups but in the setting of $G$-theory.

\subsection{$\mathbb{A}^1$-homotopic finite and Chow correspondences}
Our goal now is to prove the following:

\begin{prop}\label{ht}
Let $X$ and $Y$ be smooth proper varieties. Suppose that finite correspondences $\alpha,\beta\in \mathrm{Cor}(X,Y)$ are $\mathbb{A}^1$-homotopic. Then $\alpha,\beta$ are $\mathbb{A}^1$-homotopic as Chow correspondences.
\end{prop}

As an immediate corollary, we find that $\mathbb{A}^1$-homotopic finite correspondences induce identical pushforwards on Chow groups.

\begin{cor}\label{a1maps}
Let $X$ and $Y$ be two smooth and proper varieties. Suppose that two finite correspondences $\alpha,\beta\in \mathrm{Cor}(X,Y)$ are $\mathbb{A}^1$-homotopic. Then the morphisms \[\alpha_*,\beta_*:\CH(X)\rightarrow \CH(Y)\] induced by the cycle classes of $\alpha,\beta$ via \emph{(no.\ref{it: morph})} are equivalent.
\end{cor}

\begin{proof}
The morphism induced by $\alpha$ depends only on its rational equivalence class. By Lemma \ref{a1chow} and Proposition \ref{ht}, this is the same as the rational equivalence class of $\beta$.
\end{proof}

Before giving the proof of Proposition \ref{ht}, we need a technical lemma.

\begin{lem}\label{exists}
Let $\alpha,\beta \in\mathrm{Cor}(X,Y)$ be finite correspondences. Suppose that $h\in \mathrm{Cor}(X\times \mathbb{A}^1,Y)$ is a finite correspondence satisfying \[ \alpha= h\circ \Gamma_{i_0} \quad \mbox{and} \quad \beta= h\circ \Gamma_{i_1}\] where \[i_0:X=X\times \{0\}\rightarrow X\times \mathbb{A}^1 \quad \mbox{and} \quad i_1:X=X\times \{1\}\rightarrow X\times \mathbb{A}^1\] are the respective inclusions. Denote by \[\tilde{i_\infty}: X=X\times \{\infty\}\rightarrow X\times \mathbb{P}^1\quad \mbox{and} \quad \varphi:X\times \mathbb{A}^1\rightarrow X\times \mathbb{P}^1\] the canonical closed immersion and open complement. Write $\tilde{i_0}=\varphi\circ i_0$ and $\tilde{i_1}=\varphi\circ i_1$ for the corresponding compositions. Let \[\pi_1:X \times (X\times \mathbb{A}^1)\times Y\rightarrow X\times Y,\quad \pi_2:X\times (X\times \mathbb{A}^1\times Y)\rightarrow X\times \mathbb{A}^1\times Y,\]\[\tilde{\pi}_1:X\times (X\times \mathbb{P}^1)\times Y\rightarrow X\times Y,\quad \mbox{and}\quad \tilde{\pi}_2:X\times (X\times \mathbb{P}^1\times Y)\rightarrow X\times \mathbb{P}^1\times Y\] denote the outside and rightmost projections respectively. Then there exists a cycle $\tilde{h}$ on $X\times \mathbb{P}^1\times Y$ with the following properties:
\begin{enumerate}[\normalfont\indent (1)]
\item $(\varphi\times \mathrm{Id}_Y)^{-1}(\tilde{h})=h$
\item $\tilde{\pi}_2^{-1}(\tilde{h})$ and $\Gamma_{\tilde{i_0}}\times Y$ intersect properly
\item $\tilde{\pi}_2^{-1}(\tilde{h})$ and $\Gamma_{\tilde{i_1}}\times Y$ intersect properly
\item $\alpha= \tilde{\pi}_{1*}(\tilde{\pi}_2^{-1}(\tilde{h})\cdot \Gamma_{\tilde{i_0}}\times Y)$
\item $\beta=\tilde{\pi}_{1*}(\tilde{\pi}_2^{-1}(\tilde{h})\cdot \Gamma_{\tilde{i_1}}\times Y)$.
\end{enumerate}
\end{lem}

\begin{proof}
Each of the properties (1)-(5) are on the cycle level. Thus, it suffices to assume that $h$ is an irreducible (or prime) cycle corresponding to an integral closed subscheme $h=U$. A candidate for $\tilde{h}$ is easy to find: set $\tilde{h}=\bar{U}$ to be the closure of $U$ in $X\times \mathbb{P}^1\times Y$ with the reduced induced closed subscheme structure. Property (1) then follows immediately. The difficulty of the proof is in checking properties (2)-(5).

We're going to show properties (2) and (4), with properties (3) and (5) having the same proof (modulo making some changes in notation). For this, it suffices to prove the equality (of schemes)\ite{it: eq}\[\tag{no.\ref{it: eq}}\tilde{\pi}_2^{-1}(\tilde{h})\cap (\Gamma_{\tilde{i_0}}\times Y)= (\mathrm{Id}_X \times \varphi \times \mathrm{Id}_Y)(\pi_2^{-1}(h)\cap \Gamma_{i_0}\times Y).\] Indeed, one has \[\mathrm{dim}(\Gamma_{\tilde{i_0}}\times Y)=\mathrm{dim}(\Gamma_{i_0}\times Y)=\mathrm{dim}(X\times Y)\quad \mbox{and} \quad \mathrm{dim}(I) \geq \mathrm{dim}(\pi_2^{-1}(h)),\] for any irreducible component $I\subset \tilde{\pi}_2^{-1}(\tilde{h})$, which shows that (2) follows from (no.\ref{it: eq}). Now, from (2) and the equality (of cycles) 
\begin{align*}\tilde{\pi}_{1*}(\tilde{\pi}_2^{-1}(\tilde{h})\cdot \Gamma_{\tilde{i_0}}\times Y) &= \tilde{\pi}_{1*}\left((\mathrm{Id}_X \times \varphi \times \mathrm{Id}_Y)_*(\pi_2^{-1}(h)\cdot \Gamma_{i_0}\times Y)\right)\\ &= \pi_{1*}(\pi_2^{-1}(h)\cdot \Gamma_{i_0}\times Y)\\ &=\alpha,\end{align*} one sees that (4) also follows from (no.\ref{it: eq}).

The lemma can then be completed as follows. We have equalities
\begin{align*}
\pi_2^{-1}(h)\cap (\Gamma_{i_0}\times Y) & = \pi_2^{-1}((\varphi\times \mathrm{Id}_Y)^{-1}(\tilde{h}))\cap (\Gamma_{i_0}\times Y)\\
& = (\mathrm{Id}_X \times \varphi \times \mathrm{Id}_Y)^{-1}(\tilde{\pi}_2^{-1}(\tilde{h})) \cap (\Gamma_{i_0}\times Y)
\end{align*}
where the first line comes from property (1), and the second line is by commutativity of the appropriate maps. Hence there is an equality
\begin{align*}
(\mathrm{Id}_X \times \varphi \times \mathrm{Id}_Y)(\pi_2^{-1}(h)\cap (\Gamma_{i_0}\times Y)) 
&= \tilde{\pi}_2^{-1}(\tilde{h}) \cap (\mathrm{Id}_X \times \varphi \times \mathrm{Id}_Y)(\Gamma_{i_0}\times Y)\\
&=\tilde{\pi}_2^{-1}(\tilde{h})\cap (\Gamma_{\tilde{i_0}}\times Y)
\end{align*}
since $\mathrm{Id}_X\times \varphi \times \mathrm{Id}_Y$ is an open immersion.
\end{proof}

\begin{proof}[Proof of Proposition \ref{ht}]
We're going to show that the finite correspondences $\alpha$ and $\beta$ define the same cycle class in $\CH_{\mathrm{dim}(X)}(X\times Y)$ under the map (no.\ref{it: corcor}). Similar to Lemma \ref{exists}, we write \[i_0:X= X\times \{0\} \rightarrow X\times \mathbb{A}^1, \quad i_1:X= X\times \{1\}\rightarrow X\times \mathbb{A}^1,\quad \mbox{and} \quad \tilde{i_\infty}:X=X\times\{\infty\}\rightarrow X\times \mathbb{P}^1 \] for the corresponding closed immersions and we write \[\varphi:X\times \mathbb{A}^1\rightarrow X\times \mathbb{P}^1\] for the open complement of $\tilde{i_\infty}$. We set $\tilde{i_0}=\varphi\circ i_0$ and $\tilde{i_1}=\varphi\circ i_1$ to be the compositions. We denote by \[\pi_P:X\times \mathbb{P}^1\times Y\rightarrow X\times Y\quad \mbox{and} \quad \pi_A: X\times \mathbb{A}^1\times Y \rightarrow X\times Y\] the respective projections. 

For the remainder of this proof, we're going to refer to the following diagram. \[\begin{tikzcd}
 & \CH(X\times Y)\arrow[swap,"(\tilde{i_\infty} \times \mathrm{Id}_Y)_*"]{d} & \\
 \CH(X\times Y)\arrow["\pi_P^*"]{r}\arrow[equals]{d} & \CH(X\times\mathbb{P}^1\times Y)\arrow[swap, "(\varphi\times \mathrm{Id}_Y)^*"]{d}\arrow[swap, shift right, "(\tilde{i_1}\times \mathrm{Id}_Y)^*"]{r}\arrow[shift left, "(\tilde{i_0}\times \mathrm{Id}_Y)^*"]{r} & \CH(X\times Y)\arrow[equals]{d}\\
 \CH(X\times Y)\arrow["\pi^*_A"]{r} & \CH(X\times \mathbb{A}^1\times Y)\arrow{d}\arrow[swap, shift right, "(i_1\times \mathrm{Id}_Y)^*"]{r}\arrow[shift left, "(i_0\times \mathrm{Id}_Y)^*"]{r} & \CH(X\times Y)\\
 & 0 & 
\end{tikzcd}\] Here the middle column is the exact localization sequence associated to the closed immersion and open complement $(\tilde{i_\infty}\times \mathrm{Id}_Y,\varphi\times \mathrm{Id}_Y)$; the left square of this diagram, with every map a flat pullback, is commutative; the right square of this diagram, with horizontal arrows the Gysin pullbacks along regular embeddings, is commutative in two different ways.

Suppose that $h\in \mathrm{Cor}(X\times \mathbb{A}^1,Y)$ is a finite correspondence realizing the $\mathbb{A}^1$-homotopy (in the sense of Definition \ref{htdef}) between the finite correspondences $\alpha$ and $\beta$. Let $\tilde{h}$ be the cycle on $X\times \mathbb{P}^1\times Y$ constructed from $h$ as in the proof of Lemma \ref{exists}. By property (1) of Lemma \ref{exists}, it follows that $(\varphi\times \mathrm{Id}_Y)^*(\tilde{h})=h$.

As $\pi_A$ is an affine bundle, the flat pullback $\pi_A^*$ is surjective. Therefore there is an element $x\in \CH(X\times Y)$ with $\pi^*_A(x)=h$. Let $y=\pi_P^*(x)$ so that we can write $$\pi_P^*(x)=\tilde{h}+(y-\tilde{h})=\tilde{h}+(\tilde{i_\infty}\times \mathrm{Id}_Y)_*(z)$$ for some element $z\in \CH(X\times Y)$ with $(\tilde{i_\infty}\times \mathrm{Id}_Y)_*(z)=y-\tilde{h}$.

By functorality of pullbacks, we have \[(\tilde{i_0}\times \mathrm{Id}_Y)^*\circ \pi_P^* = \mathrm{Id}_{X\times Y}^* = (\tilde{i_1}\times \mathrm{Id}_Y)^*\circ \pi_P^*.\] From \cite[Proposition 55.3]{MR2427530} it follows \[(\tilde{i_0}\times \mathrm{Id}_Y)^*\circ (\tilde{i_\infty}\times \mathrm{Id}_Y)_* = 0 = (\tilde{i_1}\times \mathrm{Id}_Y)^*\circ (\tilde{i_\infty}\times \mathrm{Id}_Y)_*.\]

Putting everything above together we find
\begin{align*}
(\tilde{i_0}\times \mathrm{Id}_Y)^*(\tilde{h}) & = (\tilde{i_0}\times \mathrm{Id}_Y)^*\left(\tilde{h}+(\tilde{i_\infty}\times \mathrm{Id}_Y)_*(z)\right)\\ & = (\tilde{i_0}\times \mathrm{Id}_Y)^*\circ \pi_P^*(x)\\
& =  (\tilde{i_1}\times \mathrm{Id}_Y)^*\circ \pi_P^*(x)\\
& = (\tilde{i_1}\times \mathrm{Id}_Y)^*\left(\tilde{h}+(\tilde{i_\infty}\times \mathrm{Id}_Y)_*(z)\right)\\
&= (\tilde{i_1}\times \mathrm{Id}_Y)^*(\tilde{h}).
\end{align*}
But the cycle classes of $\alpha,\beta$ in $\CH_{\mathrm{dim}(X)}(X\times Y)$ are determined by the Gysin pullbacks \[\alpha = (\tilde{i_0}\times \mathrm{Id}_Y)^*(\tilde{h})\quad \mbox{and} \quad \beta=(\tilde{i_1}\times \mathrm{Id}_Y)^*(\tilde{h})\] because, on the level of cycle classes, one has \[\alpha=\tilde{h} \circ \Gamma_{\tilde{i_0}}\quad \mbox{and} \quad \beta=\tilde{h} \circ \Gamma_{\tilde{i_1}}\] by properties (4) and (5) of Lemma \ref{exists}. One gets the result by comparing with (no.\ref{it: morph2}), (no.\ref{it: pull}), and (no.\ref{it: gys}) above.
\end{proof}

\begin{exmp}
Let $\mathbb{P}^1$ have coordinates $x,y$ and let $\mathbb{P}^2$ have coordinates $X,Y,Z$. Consider the graph $h=\Gamma_{g(t)}$ of the morphism \[g(t):\mathbb{P}^1\times \mathbb{A}^1\rightarrow \mathbb{P}^2\] defined by the rule \[([x:y], t)\mapsto [x^2:txy:y^2].\] Let $i_0:\mathbb{P}^1\rightarrow \mathbb{P}^1\times \mathbb{A}^1$ be the inclusion at $0$ and $i_1:\mathbb{P}^1\rightarrow \mathbb{P}^1\times \mathbb{A}^1$ be the inclusion at $1$. It's not difficult to check that $g(t)\circ i_0:\mathbb{P}^1\rightarrow \mathbb{P}^2$ is a double cover of the line $Y=0$ while $g(t)\circ i_1:\mathbb{P}^1\rightarrow \mathbb{P}^2$ is the standard Veronese embedding. Then the morphisms \[(g(t)\circ i_1)_*, (g(t)\circ i_0)_*: \CH(\mathbb{P}^1)\rightarrow \CH(\mathbb{P}^2)\] are the same by Corollary \ref{a1maps} because $h$ is an explicit $\mathbb{A}^1$-homotopy between $g(t)\circ i_0$ and $g(t)\circ i_1$. This could also be checked directly since \[(g(t)\circ i_1)_*, (g(t)\circ i_0)_*:\CH_0(\mathbb{P}^1)=\mathbb{Z}\xrightarrow{\sim} \mathbb{Z}=\CH_0(\mathbb{P}^2)\] is an isomorphism and \[(g(t)\circ i_1)_*, (g(t)\circ i_0)_*:\CH_1(\mathbb{P}^1)=\mathbb{Z}\subsetneq \mathbb{Z}=\CH_1(\mathbb{P}^2)\] is the inclusion with image of index 2.
\end{exmp}

\begin{cor}\label{homiso}
Let $X$ and $Y$ be two smooth and proper varieties. Suppose that $X$ and $Y$ are naively $\mathbb{A}^1$-homotopy equivalent by morphisms $f:X\rightarrow Y$ and $g:Y\rightarrow X$. Then the morphisms \[f_*:\CH(X)\rightarrow \CH(Y) \quad \mbox{and} \quad g_*:\CH(Y)\rightarrow \CH(X)\] are mutually inverse isomorphisms.
\end{cor}

\begin{proof}
By (no.\ref{it: push}) and (no.\ref{it: id}) we find $\Delta_{X*}, \Delta_{Y*}$ are the identities. By (no.\ref{it: gc}), \cite[Proposition 62.8]{MR2427530}, and Proposition \ref{ht} there are equalities \[\Gamma_{g\circ f *} = (\Gamma_g \circ \Gamma_f)_{\ast} = \Gamma_{g*}\circ\Gamma_{f*}= \Delta_{X*}\] and \[\Gamma_{f\circ g *} = (\Gamma_f \circ \Gamma_g)_{\ast} = \Gamma_{f*}\circ\Gamma_{g*}= \Delta_{Y*}\] which completes the proof.
\end{proof}

\subsection{$\mathbb{A}^1$-$G$-homotopy equivalences}
In this subsection we extend our results to $G$-theory in analogy with the results that have been obtained above. Using a substantially simplified version of homotopy equivalence that allows us to work with $G$-correspondences, the process turns out to be much easier.

\begin{defn}
We say that two varieties $X$ and $Y$ are \textit{naively $\mathbb{A}^1$-$G$-homotopy equivalent} if there exist morphisms $f:X\rightarrow Y$ and $g:Y\rightarrow X$ such that both the class of the structure sheaf $\mathcal{O}_{\Gamma_{g\circ f}}$ and the class of $\mathcal{O}_{\Delta_X}$ are $\mathbb{A}^1$-homotopic $G$-correspondences in $G(X\times X)$ and, the class of $\mathcal{O}_{\Gamma_{f\circ g}}$ and the class of $\mathcal{O}_{\Delta_Y}$ are $\mathbb{A}^1$-homotopic as $G$-correspondences in $G(Y\times Y)$.
\end{defn}

The following corollary is then immediate from Lemma \ref{a1k}.

\begin{cor}\label{htk}
Let $X$ and $Y$ be two smooth and proper varieties. Suppose that $\alpha,\beta \in G(X\times Y)$ are two $\mathbb{A}^1$-homotopic $G$-correspondences. Then the morphisms \[ \alpha_*,\beta_*:G(X)\rightarrow G(Y)\] induced by $\alpha,\beta$ via \emph{(no.\ref{it: kmorph})} are equivalent.$\hfill\square$
\end{cor}

Similarly, we get an analog of Corollary \ref{homiso} that holds in $G$-theory with basically the exact same proof as before (modulo notational changes).

\begin{cor}\label{cork}
Let $X$ and $Y$ be two smooth and proper varieties. Suppose that $X$ and $Y$ are naively $\mathbb{A}^1$-$G$-homotopy equivalent by morphisms $f:X\rightarrow Y$ and $g:Y\rightarrow X$. Then the morphisms \[f_*:G(X)\rightarrow G(Y) \quad \mbox{and}\quad g_*:G(Y)\rightarrow G(X)\] are mutually inverse isomorphisms.$\hfill\square$
\end{cor}

\section{Analogs of Whitehead's theorem}\label{s: wht}
In this section, we conclude with analogs of Whitehead's theorem from algebraic topology. Recall that Whitehead's theorem (see \cite{MR30760} or \cite[Theorem 4.5]{MR1867354} for a more recent treatment) is the following statement: if a continuous map $f:X\rightarrow Y$, between connected CW complexes $X$ and $Y$, induces an isomorphism on homotopy groups $f_*:\pi_n(X)\rightarrow \pi_n(Y)$ for every $n\geq 1$, then $f$ is a homotopy equivalence. Combined with Hurewicz's theorem, Whitehead's theorem implies that every continuous map $f:X\rightarrow Y$, between simply connected CW complexes $X$ and $Y$, that induces an isomorphism on singular homology $f_*:H_n(X, \mathbb{Z})\rightarrow H_n(Y,\mathbb{Z})$ for all $n\in \mathbb{Z}$ is a homotopy equivalence. The first analog of Whitehead's theorem that we prove is directly related to this latter statement:

\begin{thm}\label{wht}
	Let $X$ and $Y$ be smooth and projective varieties. Suppose that there is a morphism $f:X\rightarrow Y$ such that the proper pushforward \[f_*:\CH(X)\rightarrow \CH(Y)\] is an isomorphism. Then $f$ is an isomorphism.
\end{thm}

Indeed, in the course of the proof we'll also show:

\begin{cor}\label{sing}
Let $X$ and $Y$ be smooth and projective varieties over $\mathbb{C}$. Let $f:X\rightarrow Y$ be a morphism of varieties. If, for all even integers $n\in \mathbb{Z}$, the map $f$ induces isomorphisms \[f_*:H_n(X(\mathbb{C}),\mathbb{Z})\rightarrow H_n(Y(\mathbb{C}),\mathbb{Z})\] between the singular homology of the associated complex manifolds $X(\mathbb{C})$ and $Y(\mathbb{C})$, then $f$ is an isomorphism.
\end{cor}

This corollary directly strengthens Whitehead's theorem, in the restricted setting of complex varieties and morphisms between them, by no longer requiring the simply connected condition and by showing that $f$ is an isomorphism rather than a homotopy equivalence.

Lastly, we show a second analog of Whitehead's theorem that holds for $G$-theory.

\begin{thm}\label{whtk}
	Let $X$ and $Y$ be smooth and projective varieties. Suppose that there is a morphism $f:X\rightarrow Y$ such that the proper pushforward \[f_*:G(X)\rightarrow G(Y)\] is an isomorphism. Then $f$ is an isomorphism.
\end{thm}

The strategy of proof for both Theorem \ref{wht} and Theorem \ref{whtk} are the same. In both cases we determine what injectivity, or surjectivity, on the given theory implies for morphisms between arbitrary, not necessarily smooth, projective varieties. By combining these results in the case of smooth varieties and applying a version of Zariski's Main Theorem, we deduce that any morphism having these properties is an isomorphism.

\subsection{Whitehead's theorem for Chow groups}\label{whtchow}
We start with some observations on the Chow groups of projective varieties.

\begin{lem}\label{tors}
	Let $X$ be a projective variety. Then $\CH_i(X)\neq 0$ if and only if $0\leq i \leq \mathrm{dim}(X)$. Moreover, if $V$ is an integral dimension-$i$ subscheme of $X$, then the class $[V]\in \CH_i(X)$ is nonzero and not torsion.
\end{lem}

\begin{proof}
	Clearly $\CH_i(X)$ is nonzero only if $0\leq i \leq \mathrm{dim}(X)$. On the other hand, let $V\subset X$ be an integral dimension-$i$ subscheme of $X$. Let $j:X\rightarrow \mathbb{P}^n$ be an embedding. Then $j_*([V])=[j(V)]$ is the class of an effective cycle in $\CH_i(\mathbb{P}^n)$ and by \cite[Example 2.5.2]{MR1644323} this class is nonzero and not torsion (it is the positive multiple $\mathrm{deg}(j(V))[L]$ of the class of a dimension-$i$ linear space $L$).
\end{proof}

\begin{lem}\label{fin}
Let $f:X\rightarrow Y$ be a morphism between projective varieties $X$ and $Y$. If the kernel of the pushforward $f_*:\CH(X)\rightarrow \CH(Y)$ is torsion, then $f$ is finite.
\end{lem}

\begin{proof}
Assume otherwise, that there is a positive dimensional subscheme $V$ that has image a point in $Y$. Then $f_*([V])=0$ but, $[V]\neq 0$ and $[V]$ is not torsion by Lemma \ref{tors}.
\end{proof}

\begin{rmk}\label{replace}
If the base field $k=\mathbb{C}$ and if $X$ is a smooth projective complex variety, then the statement of Lemma \ref{tors} can be modified to hold for the singular homology groups $H_*(X(\mathbb{C}),\mathbb{Z})$ with degrees appropriately scaled by 2. More precisely, for every $0\leq i \leq \mathrm{dim}(X)$ the group $H_{2i}(X(\mathbb{C}),\mathbb{Z})$ is nonzero; moreover, if $V$ is a projective, complex, and dimension-$i$ subvariety of $X$ then the pushforward of the cycle class $[V]$ along the inclusion $V\subset X$ is nonzero and nontorsion. This is because pushforward maps commute with the cycle class maps from the Chow groups by \cite[Lemma 19.1.2]{MR1644323} so that one can again reduce to the case $X=\mathbb{P}^n$.
\end{rmk}

\begin{rmk}\label{replace2}
Lemma \ref{fin} also holds in the setting of singular homology by the same reasoning. If $f:X\rightarrow Y$ is a morphism between smooth projective complex varieties $X$ and $Y$ and, if the fiber $V=f^{-1}(y)$ over a closed point $y\in Y$ is positive dimensional, then one has an equality $f_*([V])=0$ inside $H_{2*}(Y(\mathbb{C}),\mathbb{Z})$ since the cycle class commutes with pushforwards. But, the cycle class of $V$ in $H_{2*}(X(\mathbb{C}), \mathbb{Z})$ is nonzero and nontorsion because of Remark \ref{replace}.
\end{rmk}

\begin{lem}\label{bir}
Let $f:X\rightarrow Y$ be a morphism between projective varieties $X$ and $Y$. Suppose that the pushforward $f_*:\CH(X)\rightarrow \CH(Y)$ is a surjection. Then the following statements hold:
\begin{enumerate}[\indent\normalfont (1)]
\item $f$ is a surjection
\item if $f$ is moreover finite, then $f$ is birational.
\end{enumerate}
\end{lem}

\begin{proof}
Let $V\subset Y$ be a nonempty open subscheme and let $U=f^{-1}(V)$. Then the cartesian square on the left below with horizontal arrows the canonical inclusions, induces the commuting square on the right below by \cite[Proposition 1.7]{MR1644323}.\[\begin{tikzcd}
U\arrow["j"]{r}\arrow["f|_U"]{d} & X \arrow["f"]{d} & & \CH(X)\arrow["j^*"]{r}\arrow["f_*"]{d} & \CH(U)\arrow["(f|_U)_*"]{d}\\ V\arrow["i"]{r} & Y & & \CH(Y)\arrow["i^*"]{r} & \CH(V)
\end{tikzcd}\]
Since $i^*$ is surjective by localization and $f_*$ is surjective by assumption, it follows that $(f|_U)_*$ is surjective for every nonempty open subscheme $V\subset Y$. Hence $f$ is surjective (if $f$ were not surjective, then $f^{-1}(Y\setminus f(X))=\emptyset$ and $\CH(\emptyset)=0$, but $\CH(Y\setminus f(X))\neq 0$). This proves (1).

To prove (2), assume that $f$ is finite. As $f$ is both finite and surjective, $X$ and $Y$ have the same dimension. Hence the surjective map \[f_*:\mathbb{Z}=\mathrm{CH}_{\mathrm{dim}(X)}(X)\rightarrow \CH_{\mathrm{dim}(Y)}(Y)=\mathbb{Z},\] where $f_*([X])=\mathrm{deg}(f)[Y]$, is an isomorphism. It follows that $\mathrm{deg}(f)=1$ and $f$ is birational. 
\end{proof}

We turn our attention to smooth and projective varieties for the proof of Theorem \ref{wht}. 

\begin{proof}[Proof of Theorem \ref{wht}]
Our assumptions and Lemma \ref{fin} show that $f$ is finite. Our assumptions and Lemma \ref{bir} show that $f$ is birational. The theorem then follows from Zariski's Main Theorem \cite[Chapter 4, Corollary 4.6]{liu2006algebraic} that says any finite and birational morphism between smooth and projective varieties is an isomorphism.
\end{proof}

A similar proof works for Corollary \ref{sing}:

\begin{proof}[Proof of Corollary \ref{sing}]
Note that $H_j(X(\mathbb{C}),\mathbb{Z})$ vanishes for $j<0$ trivially, and for $j>2\mathrm{dim}(X)$ by Poincar\'e duality. By applying the first part of Remark \ref{replace} to this observation one gets the equality $\mathrm{dim}(X)=\mathrm{dim}(Y)$. By Remark \ref{replace2} the map $f$ is finite, hence dominant. The isomorphism \[f_*:\mathbb{Z}=H_{2\mathrm{dim}(X)}(X(\mathbb{C}),\mathbb{Z})\rightarrow H_{2\mathrm{dim}(Y)}(Y(\mathbb{C}),\mathbb{Z})=\mathbb{Z}\] has the property that $f_*([X])=\mathrm{deg}(f)[Y]$ by \cite[Lemma 19.1.2]{MR1644323}, showing that $f$ is birational. Therefore, by Zariski's Main Theorem \cite[Chapter 4, Corollary 4.6]{liu2006algebraic}, $f$ is an isomorphism.
\end{proof}

We conclude with an application showing that naively $\mathbb{A}^1$-homotopy equivalent smooth and projective varieties are isomorphic.

\begin{cor}\label{last}
	Assume that $X$ and $Y$ are naively $\mathbb{A}^1$-homotopy equivalent smooth and projective varieties. Then $X$ and $Y$ are isomorphic.
\end{cor}

\begin{proof}
	Apply Theorem \ref{wht} and Corollary \ref{homiso}.
\end{proof}

\begin{exmp}\label{exmp11}
	As the lowest dimensional nonexample of Theorem \ref{wht}, we remark that a smooth conic curve $C$ over $\mathbb{Q}$ without $\mathbb{Q}$-rational points has Chow groups \[\CH_0(C)=\mathbb{Z}\quad  \mbox{and} \quad \CH_1(C)=\mathbb{Z}\] so that $\CH(C)$ and $\CH(\mathbb{P}^1)$ are abstractly isomorphic but, there are no morphisms $f:\mathbb{P}^1\rightarrow C$ (constant or otherwise) so, $\mathbb{P}^1$ and $C$ are never naively $\mathbb{A}^1$-homotopy equivalent.
\end{exmp}

\subsection{Whitehead's theorem for G-theory}\label{whtkth}
We work throughout this section with the $G$-theory of a variety as developed in \cite{MR0265355}. We use the notation $\tau_i(X)$ to denote the $i$th piece of the topological filtration on the Grothendieck group $G(X)$ of coherent sheaves on a variety $X$. By definition, $\tau_i(X)$ is the subgroup of $G(X)$ generated by classes of coherent sheaves supported in dimension-$i$ or less. We start with some observations on this filtration and on the $G$-theory of projective varieties.

\begin{lem}\label{torsk}
Let $X$ be a projective variety. Then the quotient $\tau_i(X)/\tau_{i-1}(X)\neq 0$ is nonzero if and only if $0\leq i \leq \mathrm{dim}(X)$. Moreover, if $V$ is an integral dimension-$i$ subscheme of $X$ and if $x$ is any closed point of $X$, then both of the following hold:
\begin{enumerate}[\indent\normalfont (1)]\item the class $[\mathcal{O}_V]$ is nonzero and not torsion in $\tau_i(X)/\tau_{i-1}(X)$
\item the difference of classes $m[\mathcal{O}_V]-n[\mathcal{O}_x]\in \tau_i(X)$ is nonzero and not torsion for every pair of integers $m,n\in \mathbb{Z}$ with $m\neq 0$.\end{enumerate}
\end{lem}

\begin{proof}
Clearly $\tau_i(X)/\tau_{i-1}(X)\neq 0$ only if $0\leq i \leq \mathrm{dim}(X)$. Conversely, let $j:X\rightarrow \mathbb{P}^r$ be a closed embedding and let $L\subset \mathbb{P}^r$ be linear subspace of dimension $r-i$. Then the element $m[\mathcal{O}_V]-n[\mathcal{O}_x]\in \tau_i(X)$ of the lemma statement is nonzero and not torsion since $$j_*(m[\mathcal{O}_V]-n[\mathcal{O}_x])\cdot [\mathcal{O}_L]=(m[\mathcal{O}_{j(V)}]-n[\mathcal{O}_{j(x)}])\cdot[\mathcal{O}_L]=m\mathrm{deg}(j(V))[\mathcal{O}_p]$$ for any rational point $p\in \mathbb{P}^r$. As this equality also holds in the group $\tau_i(\mathbb{P}^r)/\tau_{i-1}(\mathbb{P}^r)$ associated to the topological filtration on $G(\mathbb{P}^r)$, we see that $\tau_i(X)/\tau_{i-1}(X)\neq 0$ if $0\leq i \leq \mathrm{dim}(X)$ as well. 
\end{proof}

\begin{lem}\label{fink}
Let $f:X\rightarrow Y$ be a morphism between projective varieties $X$ and $Y$. If the kernel of the pushforward $f_*:G(X)\rightarrow G(Y)$ is torsion, then $f$ is finite.
\end{lem}

\begin{proof}
Suppose otherwise that there is some positive dimensional integral subscheme $V\subset X$ with $f(V)=y$ a closed point of $Y$. Let $x$ be any closed point of $V$. Then $f_*[\mathcal{O}_V]=n[\mathcal{O}_y]$ and $f_*([\mathcal{O}_x])=m[\mathcal{O}_y]$ where $n=(f|_V)_*([\mathcal{O}_V])=\chi(\mathcal{O}_V)/\mathrm{deg}(k(y)/k)$ is the pushforward, along the map $f|_V:V\rightarrow y$, of $[\mathcal{O}_V]\in G(V)$ and $m=\mathrm{deg}(k(x)/k(y))\neq 0$. This implies, in particular, that \[f_*(m[\mathcal{O}_V]-n[\mathcal{O}_x])=mn[\mathcal{O}_y]-mn[\mathcal{O}_y]=0\] but, $m[\mathcal{O}_V]-n[\mathcal{O}_x]$ is a nonzero and not torsion element of $\tau_i(X)\subset G(X)$ by Lemma \ref{torsk} (2).
\end{proof}

\begin{lem}\label{birk}
	Let $f:X\rightarrow Y$ be a morphism between projective varieties $X$ and $Y$. Suppose that the pushforward $f_*:G(X)\rightarrow G(Y)$ is a surjection. Then the following statements hold:
	\begin{enumerate}[\indent\normalfont (1)]
		\item $f$ is a surjection
		\item if $f$ is moreover finite, then $f$ is birational.
	\end{enumerate}
\end{lem}

\begin{proof}
Let $i:V\rightarrow Y$ be the inclusion of an open dense subscheme and let $j:U=f^{-1}(V)\rightarrow X$ be the corresponding inclusion of the preimage. The diagram on the left of (no.\ref{it: ksq}) below is cartesian. The diagram on the right of (no.\ref{it: ksq}) below commutes by applying flat base change \cite[Chapter 3, Proposition 9.3]{MR0463157} to the diagram on the left.
\ite{it: ksq}\[\tag{no.\ref{it: ksq}}\begin{tikzcd}
U\arrow{r}\arrow["f|_U"]{d} & X \arrow["f"]{d} & & G(X)\arrow["j^*"]{r}\arrow["f_*"]{d} & G(U)\arrow["(f|_U)_*"]{d}\\ V\arrow{r} & Y & & G(Y)\arrow["i^*"]{r} & G(V)
\end{tikzcd}\]
The horizontal arrows in the right square of (no.\ref{it: ksq}) are surjective because of localization. Since $f_*$ is a surjection by assumption, we get that $(f|_U)_*$ is a surjection as well. In particular, this implies that $f$ is surjective (if $f$ were not surjective, then setting $V=Y\setminus f(X)$ gives $U=f^{-1}(Y\setminus f(X))=\emptyset$ and $G(\emptyset)=0$ whereas $G(Y\setminus f(X))\neq 0$) which proves (1).

To prove (2), choose a nonempty open $V\subset Y$ with preimage $U=f^{-1}(V)$ flat over $V$; such a $V$ exists as a consequence of generic flatness, see \cite[\href{https://stacks.math.columbia.edu/tag/052A}{Proposition 052A}]{stacks-project}. Since $f$ is finite, and therefore also $f|_U$ is finite, we can apply Grauert's theorem \cite[Chapter 3, Corollary 12.9]{MR0463157} to conclude that $(f|_U)_* \mathcal{O}_U$ is a locally free sheaf on $V$ of rank $\mathrm{deg}(f)$. We're going to show that $\mathrm{rk}(f_*\mathcal{O}_X)=1$ which will complete the proof since $\mathrm{rk}(f_*\mathcal{O}_X)=\mathrm{rk}((f|_U)_*\mathcal{O}_U)$.

Taking limits over the diagram on the right of (no.\ref{it: ksq}) for all nonempty open subschemes $V\subset Y$, we get the commuting square
\ite{it: ksqlim}\[\tag{no.\ref{it: ksqlim}}\begin{tikzcd}
G(X)\arrow{r}\arrow["f_*"]{d} & \varinjlim_{f^{-1}(V)} G(f^{-1}(V))\arrow{d} &\arrow[equals]{l} G(\eta_X) = \mathbb{Z} \\
G(Y)\arrow{r} & \varinjlim_V G(V)  & \arrow[equals]{l} G(\eta_Y) = \mathbb{Z}
\end{tikzcd}\] where $\eta_X$ and $\eta_Y$ are the generic points of $X$ and $Y$ respectively. The horizontal arrows of (no.\ref{it: ksqlim}) are surjective by localization. Since $f_*$ is also surjective, the right vertical arrow of (no.\ref{it: ksqlim}) is an isomorphism. Identifying the horizontal arrows of (no.\ref{it: ksqlim}) with the morphisms assigning to a coherent sheaf its rank, it follows from going around the diagram that $\mathrm{rk}(f_*\mathcal{O}_X)=1$.
\end{proof}

We turn our attention to smooth and projective varieties for the proof of Theorem \ref{whtk}.

\begin{proof}[Proof of Theorem \ref{whtk}]
Our assumptions and Lemma \ref{fink} imply that $f$ is finite. Our assumptions and Lemma \ref{birk} show that $f$ is birational. By Zariski's Main Theorem \cite[Chapter 4, Corollary 4.6]{liu2006algebraic}, we conclude that $f$ is an isomorphism. \end{proof}

We conclude with an application showing that naively $\mathbb{A}^1$-$G$-homotopy equivalent smooth and projective varieties are isomorphic.

\begin{cor}
Assume that $X$ and $Y$ are $\mathbb{A}^1$-$G$-homotopy equivalent smooth and projective varieties. Then $X$ and $Y$ are isomorphic.
\end{cor}

\begin{proof}
Apply Theorem \ref{whtk} and Corollary \ref{cork}.
\end{proof}

\begin{exmp}
The lowest dimensional nonexample of Theorem \ref{whtk} is the same as the lowest dimensional nonexample to Theorem \ref{wht} contained in Example \ref{exmp11}: if $C$ is a smooth conic over $\mathbb{Q}$ without $\mathbb{Q}$-rational points, then $G(C)=\mathbb{Z}\oplus \mathbb{Z}$ so that $G(C)$ and $G(\mathbb{P}^1)$ are abstractly isomorphic but, there are no morphisms $f:\mathbb{P}^1\rightarrow C$.
\end{exmp}

\bibliographystyle{amsalpha}
\bibliography{bib}
\end{document}